\documentclass[12pt]{amsart}

\usepackage{hyperref}
\usepackage{amsfonts}
\usepackage{amssymb}
\usepackage{amsxtra}
\usepackage{amstext}
\usepackage[english]{babel}

\usepackage{mathrsfs} 
\newtheorem{theorem}{Theorem}[section]
\newtheorem{assumption}[theorem]{Assumption}
\newtheorem{lemma}[theorem]{Lemma}

\newtheorem{proposition}[theorem]{Proposition}

\newcommand{\E}{\mathbb{E}}
\renewcommand{\P}{\mathbb{P}}

\newcommand{\R}{\mathbb{R}}

\newcommand{\N}{\mathbb{N}}
\def\eps{\varepsilon}
\def\phi{\varphi}
\newcommand{\rmd}{\,\mathrm{d}}

\def\cal{\mathcal}

\begin{document}

\title{On a class of random walks in simplexes}

\author{Tuan-Minh Nguyen}
\address[\sc T.M. Nguyen]{School of Mathematical Sciences, Monash University, Victoria 3800, Australia}
\email{tuanminh.nguyen@monash.edu}

\author{Stanislav Volkov}

\address[\sc S. Volkov]{Centre for Mathematical Sciences, Lund University, Lund 22100-118, Sweden}
\email{stanislav.volkov@matstat.lu.se}

\begin{abstract}
We study the limit behaviour of a class of random walk models taking values in the standard $d$-dimensional ($d\ge 1$) simplex. From an interior point $z$, the process chooses one of the $d+1$ vertices of the simplex, with probabilities depending on $z$, and then the particle randomly jumps to a new location $z'$ on the segment connecting $z$ to the chosen vertex. In some special cases, using properties of the Beta distribution, we prove that the limiting distributions of the Markov chain are Dirichlet. We also consider a related history-dependent random walk model in $[0,1]$ based on an urn-type scheme. We show that this random walk converges in distribution to an arcsine random variable.
\end {abstract}

\keywords{Random walks in simplexes, iterated random functions, Dirichlet distribution, stick-breaking process.}

\subjclass[2010]{60J05, 60F05}
\maketitle

\section{Introduction}
Throughout this paper the $d$-dimensional standard orthogonal simplex (see e.g.~\cite{simplex}) is defined as 
$$
\mathcal{S}_d=\left\{(z_1,z_2,\dots,z_d)\in \mathbb{R}^d :\ z_1+z_2+\dots+z_d\le 1,\ z_j\ge 0,\ j=1,2,\dots,d\right\}.
$$ 
We also denote the interior of $\mathcal{S}_d$, the Borel $\sigma$-algebra, and the Lebesgue measure on $\mathcal{S}_d$ by $\mathcal{S}_d^o$, $\mathcal{B} (\mathcal{S}_d)$, and $\lambda_d$ respectively. Let $E_0=(0,0,\dots,0)$ be the origin, and $E_1=(1,0,\dots,0)$, $E_2=(0,1,0,\dots,0)$, $\dots$, $E_d=(0,\dots,0,1)$ be the standard orthonormal basis vectors in $\mathbb{R}^d$, which are also the vertices of $\mathcal{S}_d$. 

For some initial point $Z_0\in \mathcal{S}_d$, we consider the following random iteration:
$$
Z_{n+1}= (1-\xi_n) Z_n + \xi_n \Theta_n,\qquad n=0,1,2,\dots,
$$
where 
\begin{itemize}
\item $\xi_n$, $n=0,1,2,...$, are independent copies of some random variable $\xi$ with support in $[0,1]$; 

\item $\Theta_n$, $n=0,1,2,\dots$, are discrete random vectors such that
$$
\P\left(\Theta_n= E_j\ |\ Z_0,Z_1,\dots,Z_n; \xi_n  \right) = p_j(Z_n), \qquad j=0,1,2,...,d,
$$
where $p=(p_1,p_2,\dots,p_d)$ (sometimes referred to as \textit{probability choice function}) is a given (which is the same for all $n$) mapping from $\mathcal{S}_d$ to itself such that $p_j: \mathcal{S}_d\to [0,1], j=1,2,\dots d$ are Borel measurable functions, and 
$p_0(z): = 1-\sum_{j=1}^dp_j(z)$ for all $z\in\mathcal{S}_d$.
\end{itemize}

The aforementioned model originates from the Sethuraman's construction of the Dirichlet distribution (see~\cite{Sethuraman}) for the case where $p_1,p_2,...,p_d$ are positive constants. Sethuraman proved that if 
\begin{itemize}
\item $\xi\sim \text{Beta}(1,\gamma)$, where $\gamma$ is some positive constant,
\item $\Theta$ is a discrete random vector such that $\P(\Theta=E_j)=p_j$ for $j=1,2,\dots,d$, and $p_0=1-p_1-\dots-p_d>0$, 
\item $Z\sim\text{Dirichlet}(p_1\gamma,p_2\gamma,\dots,p_d\gamma,p_0\gamma)$, and 
\item $Z, \Theta, \xi$ are jointly independent,
\end{itemize}
then 
\begin{align*}
Z\overset{d}{=} (1-\xi)Z+\xi\Theta.
\end{align*}
Here $\text{Beta}(a,b)$ denotes the usual Beta distribution with the probability density function
$$
g(x)=\frac{\Gamma(a+b)}{\Gamma(a)\Gamma(b)}x^{a-1}(1-x)^{b-1},\qquad 0<x<1,
$$ 
$\Gamma$ is the Gamma-function, and $\text{Dirichlet}(\alpha_1,\alpha_2,...,\alpha_d,\alpha_{d+1})$ denotes the Dirichlet distribution with the probability density function 
$$
f(z_1,z_2,...,z_d)=\frac{\Gamma\left( \sum_{i=1}^{d+1}\alpha_i\right)}{\prod_{i=1}^{d+1} \Gamma(\alpha_i)}\left(1-\sum_{i=1}^d z_i\right)^{\alpha_{d+1}-1}\prod_{i=1}^{d}z^{\alpha_i-1}_i,\qquad (z_1,z_2,...,z_d)\in \mathcal{S}_d^o.
$$ 
Consequently, the stationary distribution of the Markov chain $\{Z_n\}_{n\ge0}$ corresponding to the Sethuraman's model is $\text{Dirichlet}(p_1\gamma,p_2\gamma,\dots, p_d\gamma, p_0\gamma)$.  Further extensions, where $\xi\sim \text{Beta}(k,\gamma)$ for some positive integer~$k$, and~$\Theta$ has a quasi-Bernoulli distribution, were studied by Hitczenco and Letac in~\cite{Hitczenko}. 

In~\cite{Diaconis}, Diaconis and Freedman reconsidered the Sethuraman's model from the point of view of random iterated functions and also studied the case where $p(z)$ depends on $z\in \mathcal{S}_1=[0,1]$. Other models in $\mathcal{S}_1$ with various special cases of $p(z)$ and $\xi$ were studied in~\cite{McKinlay,Pacheco,Ramli}. Inspired by the work of Diaconis and Freedman, Ladjimi and Peign\'e in their recent work~\cite{Ladjimi} studied iterated random functions with place-dependent probability choice functions, and demonstrated several applications to the one-dimensional model where $\xi\sim \text{Uniform}[0,1]$, and $p(z)$ is a H\"older-continuous function in $[0,1]$. 

In~\cite{McKinlay}, McKinlay and Borovkov gave a general condition for the ergodicity of the one-dimensional Markov chain $\{Z_n\}_{n\ge0}$ in $\mathcal{S}_1$. By solving integral equations, they derived a closed-form expression for the stationary density function in the case where $\xi\sim \text{Beta}(1,\gamma)$, and $p(z)$ is a piecewise continuous function on $[0,1]$. In particular, if $p(z)=(1-c)z+b(1-z), b,c\in (0,1]$, then the stationary distribution is $\text{Beta}(b\gamma, c\gamma)$. 

The model, also known in the literature as a stick-breaking process, a stochastic give-and-take (see~\cite{DeGroot}, \cite{McKinlay}) or a Diaconis-Freedman chain (see~\cite{Ladjimi}) has many applications in other fields such as human genetics, robot coverage algorithms, random search, etc. For further discussions, we refer the reader to~\cite{DeGroot}, \cite{Ramli} and \cite{McKinlay}.

The rest of the paper is organized as follows. In Section~2, we give an extension of the ergodicity criterion of MacKinlay and Borovkov to higher dimensional simplexes under certain assumptions on~$p(z)$ and~$\xi$. In Section~3, in the case where~$\xi$ is Beta-distributed while the probability choice function~$p(z)$ linearly depend on~$z$, we prove that the limiting distribution of the chain is a Dirichlet distribution. Finally, in Section~4, we consider a history-dependent random walk model in [0,1] based on urn-type schemes. Using martingales and coupling techniques, we show that the random walk converges in distribution to an arcsine random variable.

\section{Existence of the limiting distribution}
To prove the ergodicity of the Markov chain $\{Z_n\}_{n\ge 0}$, we will make use of the following result.

\begin{proposition}[Theorems~1.3 and~2.1 in~\cite{Borovkov}] \label{conderg}
Let $Z_n, n=0,1,2,...$ be a Markov chain on a measurable state space $(\mathcal{X},\mathfrak{B})$ such that for $n\ge 1$, $\P(Z_n\in A\ |\ Z_0=z)$ is a measurable function of $z\in \mathcal{X}$ when $A\in \mathfrak{B}$ is fixed, while it is a probability measure of $A$ when $z$ is fixed. 

Then $Z_n$ is ergodic, if there exists a subset $V\in \mathfrak{B}$, $q>0$, a probability measure $\phi$ on $(\mathcal{X},\mathfrak{B})$, and some positive integer $n_0$ such that
\begin{enumerate}
\item[(a)] $\P(\tau_V<\infty\ |\ Z_0=z)=1$ for all $z\in \mathcal{X}$, where $\tau_V=\inf\{n\ge 1: Z_n\in V\}$;
\item[(b)] $\sup_{z\in V} \E\left(\tau_V\ |\ Z_0=z\right)<\infty$;
\item[(c)] $\P(Z_{n_0}\in B\ |\ Z_0=z)\ge q \phi(B)$ for all $B\in \mathfrak{B}$ and $z \in V$; 
\item[(d)] $\gcd\left\{n: \P(Z_{n}\in B\ |\ Z_0=z)\ge q \phi(B)\right\}=1$ for $z \in V$.
\end{enumerate}
Moreover, if the above conditions are fulfilled, then there exists a unique invariant measure $\mu$ such that the distribution of $Z_n$ converges to $\mu$ in total variation norm.
\end{proposition}

For each $z=(z_1,z_2\dots,z_d)\in \mathcal{S}_d$, we define $z_{0}=1-z_1-z_2-\dots-z_d$. Note that the set of all $(z_0,z_1,z_2\dots,z_d)$, where $(z_1,z_2\dots,z_d)\in {\cal S}_d$, constitutes the standard $d-$simplex in $\mathbb{R}^{d+1}$. 

\begin{assumption}\label{asum1}
There exist $\delta\in(0,\frac{1}{2^d})$ and $s,t\in(\delta^{1/d},1-\delta^{1/d})$, $s<t$ such that 
\begin{enumerate}
\item[(i)] $\P(\xi< 1-\delta):=1-\eta<1$;
\item[(ii)] there is an $\eps>0$ such that for any $1\le k\le d$ and any $0\le j_1<j_2<\dots<j_k\le d$ we have
$$
\inf_{\overset{z\in \mathcal{S}_d,}{ z_{j_1}+\dots+z_{j_k}\le \delta}}\ 
\sum_{l=1}^k p_{j_l}(z)\ge \eps;
$$ 
\item[(iii)] there is $c>0$ such that for all $B\in \mathcal{B}([0,1])$, $B\subset [s(1-t)^{d-1}-\delta,t]\cup[(1-t)^{d}-\delta,1-s]$ we have
$$
\P(\xi\in B)> c\lambda(B),
$$
where $\lambda$ is the Lebesgue measure on $[0,1].$ 
\end{enumerate}
\end{assumption}

\noindent
\textbf{Remark.} Condition (i) is quite natural in order to avoid the absorption of $Z_n$ at the boundary of $\mathcal{S}_d$. For $d=1$, the above conditions are very similar to the assumptions (E1-E2-E3) of McKinlay and Borovkov in~\cite{McKinlay}. However, in contrast to our condition (iii), McKinlay and Borovkov require that $\xi$ has a density on $[s-\delta,t]$ and $[1-t-\delta,1-s]$. Also, observe that in condition (iii) the intervals are properly defined (though they may overlap). \\

For $j=0,1,\dots,d$ define
$$
V_j=\left\{z=(z_1,\dots, z_d)\in\mathcal{S}_d:\ 1-\delta\le z_j\le 1 \right\}.
$$ 
In particular,
$$
V_0=\left\{z=(z_1,\dots, z_d)\in\mathcal{S}_d:\ \sum_{j=1}^d z_j\le \delta \right\}.
$$

\begin{figure}[h]
\setlength{\unitlength}{6cm}
\begin{center}
\begin{picture}
(1,1)(0,0)
\put(0,0){\line(0,1){1}}
\put(0,0){\line(1,0){1}}
\put(1,0){\line(-1,1){1}}
\put(0.8,0){\line(0,1){0.2}}
\put(0,0.8){\line(1,0){0.2}}
\put(0,0.2){\line(1,-1){0.2}}
\put(-0.15,0.03){$E_0$}
\put(1,0.03){$E_1$}
\put(-0.15,1){$E_2$}
\put(-0.02 , 0.03){ \line(1,0){0.17} }
\put(-0.02 , 0.06){ \line(1,0){0.14} }
\put(-0.02 , 0.09){ \line(1,0){0.11} }
\put(-0.02 , 0.12){ \line(1,0){0.08} }
\put(-0.02 , 0.15){ \line(1,0){0.05} }
\put(0.78 , 0.03){ \line(1,0){0.17} }
\put(0.78 , 0.06){ \line(1,0){0.14} }
\put(0.78 , 0.09){ \line(1,0){0.11} }
\put(0.78 , 0.12){ \line(1,0){0.08} }
\put(0.78 , 0.15){ \line(1,0){0.05} }
\put(-0.02 , 0.83){ \line(1,0){0.17} }
\put(-0.02 , 0.86){ \line(1,0){0.14} }
\put(-0.02 , 0.89){ \line(1,0){0.11} }
\put(-0.02 , 0.92){ \line(1,0){0.08} }
\put(-0.02 , 0.95){ \line(1,0){0.05} }
\put(0.12,0,1){$V_0$}
\put(0.67,0,1){$V_1$}
\put(0.08,0,7){$V_2$}
\end{picture}

\caption{Illustration of $V_j$, $j=0,1,2$ in case $d=2$.}\end{center}
\end{figure}
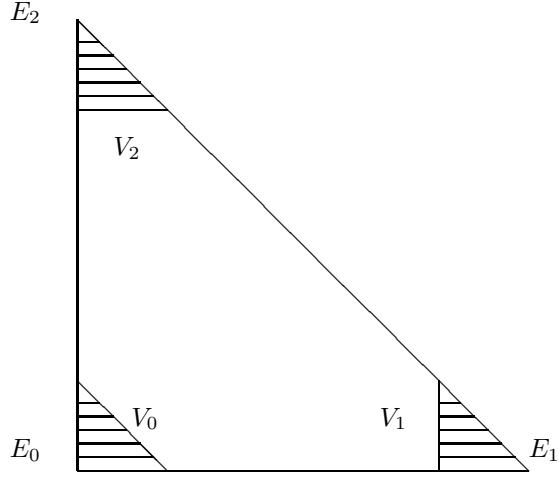

For each $x=(x_1,x_2,\dots,x_d)\in (0,1)^d$ also define $T:(0,1)^d\to\mathcal{S}_d^o$ by setting
$$
T(x)=\left(x_1\prod_{j=2}^d (1-x_j),\ x_2\prod_{j=3}^d (1-x_j),\ \dots,\ x_{d-1}(1-x_d),\ x_d\right).
$$
Note that $T$ is a homeomorphism from $(0,1)^d$ to $\mathcal{S}_d^o$, and its inverse $T^{-1}$ for each $z=(z_1,\dots,z_d)\in \mathcal{S}_d^o$ is given by 
$$
T^{-1}(z) = \left(\frac{z_1}{\displaystyle 1-\sum\limits_{2\leq j\leq d}z_j},\ \frac{z_2}{\displaystyle 1-\sum\limits_{3\leq j\leq d}z_j},\ \dots\ ,\ \frac{z_{d-1}}{1-z_d},\ z_d\right).
$$
Let $z=(z_1,\dots,z_d)$, $u=(u_1,\dots,u_d)\in\mathcal{S}_d$, $z_0=1-\sum_{k=1}^d z_k$, and $u_0=1-\sum_{k=1}^d u_k$.
For each $j=1,2,\dots,d$ we define the following functions 
\begin{align*}
R_j(u)&=(u_0,u_1,\dots,u_{j-1}, u_{j+1},u_{j+2},\dots,u_d),\qquad j=0,1,\dots,d;
\\
G_z(u)&=\left(u_0z_1+u_1,u_0z_2+u_2,\dots,u_0z_d+u_d\right).
\end{align*}
If $z_0\neq 0$ then the map $G_z$ is invertible; moreover, its inverse can be computed as 
$$
G_z^{-1}(u)=\left( u_1-\frac{z_1u_0}{z_0}, u_2-\frac{z_2u_0}{z_0},\dots,u_d-\frac{z_d u_0}{z_0} \right).
$$
For some two real numbers $s$ and $t$, such that $0<s<t<1$, define 
\begin{align*}
K&:=\left\{ (u_1,\dots,u_d)\in\mathcal{S}_d: s\le\frac{u_j}{1-\sum_{l=j+1}^d u_l}\le t, j=1,2,\dots,d \right\}
=T\left([s,t]^d\right).
\end{align*}

The proof of the following Lemma is given in the Appendix.
\begin{lemma}\label{transform}
Assume that $\delta\in\left(0,\frac{1}{2^d}\right)$, 
$s,t\in\left(\delta^{1/d},1-\delta^{1/d}\right)$ and $s<t$.
\begin{itemize}
\item[(a)] If $z\in V_0$, then
$$
G^{-1}_z(K) \subset T\left([s(1-t)^{d-1}-\delta,t]^d\right).
$$
\item[(b)] If $z\in V_k$ with $k\in \{1,2,\dots, d\}$, then
$$
G_{R_k(z)}^{-1}\circ R_k(K)\subset T\left([(1-t)^{d}-\delta,1-s]\times[s(1-t)^{d-1}-\delta,t]^{d-1}\right).
$$
\end{itemize}
\end{lemma}

\begin{theorem}\label{t1}
Assume that all the conditions in Assumption~\ref{asum1} are fulfilled. Then the Markov chain $\{Z_n\}_{n\ge0}$ converges in distribution.
\end{theorem}
\begin{proof}${}$\\
\underline{\sf Step 1.} We define 
$$
V=\bigcup_{j=0}^dV_j.
$$
From part (i) of Assumption~\ref{asum1} it follows that $\P(Z_1\in V\ |\ Z_0=z)\ge \P(\xi\ge 1-\delta)=\eta>0$ for all $z\in \mathcal{S}_d$. Therefore, for all $z\in \mathcal{S}_d$, given $Z_0=z$, the random variable $\tau_{V}=\inf\{n\ge 1:\ Z_n\in V\}$ is stochastically dominated by a geometric random variable with parameter $\eta$, thus yielding
$$
\P(\tau_V>n\ |\ Z_0=z)\le (1-\eta)^n.
$$
Hence, the conditions (a) and (b) from the statement of Proposition~\ref{conderg} are satisfied.
\\[2mm]
\underline{\sf Step 2.}
Throughout the rest of the proof, we let $\text{Const}$ denote some positive constant. From the definition of $\{Z_n\}_{n\ge 0}$, we observe that
\begin{align*}
Z_d = \zeta_0 Z_0&+ \zeta_1 \Theta_0+\zeta_2\Theta_1+\dots + \zeta_d\Theta_{d-1},
\end{align*}
where 
\begin{align*}
(\zeta_{1},\dots,\zeta_{d})&:= T(\xi_0,...,\xi_{d-1}),\\ \zeta_0 &: = \displaystyle \prod_{j=0}^{d-1}(1-\xi_j)=1-\sum_{j=1}^d \zeta_j.
\end{align*}
For $1\le k\le d$ and $0\le j_1<j_2<\dots<j_k\le d$ define
$$
U_{j_1j_2...j_{k}}:=\left\{z=(z_1,z_2,...,z_d)\in \mathcal{S}_d :\ z_{j_1}+z_{j_2}+\dots+z_{j_k}\le \delta\right\}.
$$
Let $B\in \mathcal{B}(\mathcal{S}_d)$.
Then, if $Z_0=z\in V_0$ and $\Theta_{0}=E_1, \Theta_{1}=E_2,\dots, \Theta_{d-1}=E_d$, then $Z_{j}\in U_{j+1,j+2,\dots,d}$ for $j=0,1,\dots,d$. Therefore, from part (ii) of Assumption~\ref{asum1} it follows that
\begin{equation}\label{ineq}
\begin{array}{rl}
&\P\left(Z_d\in B\ |\ Z_0=z\right)
  \ge \P\left(Z_d\in B\cap K, (\Theta_{0},\Theta_{1},\dots,\Theta_{d-1})=(E_1,E_2,\dots,E_d)\ |\ Z_0=z\right) \\
& \ge 
\left[
\prod_{l=1}^d \displaystyle \inf_{z\in U_{l,l+1,...,d}}\left(\sum_{j=l}^dp_j(z) \right)
\right]
\times
 \P\left(
\zeta_0 z + \zeta_1 E_1+\dots+\zeta_{d} E_{d} \in B\cap K \right) \\
& \ge \eps^d~ \P((\zeta_0 z_1+\zeta_1, \zeta_0 z_2+\zeta_2,\dots,\zeta_0 z_d+\zeta_d)\in B\cap K)\\
& = \eps^d~ \P\left((\zeta_1,\dots,\zeta_d)\in G^{-1}_z(B\cap K)\right)
=\eps^d~ \P\left((\xi_0,\dots,\xi_{d-1})\in T^{-1}\circ G^{-1}_z(B\cap K)\right).
\end{array}
\end{equation}
\underline{\sf Step 3.} For $B\in \mathcal{B}(\mathcal{S}_d)$ and $z\in V_0$, from part~(iii) of Assumption~\ref{asum1} and Lemma~\ref{transform}, we have
\begin{equation}\label{ineq7}
\P\left((\xi_0,\xi_1,\dots,\xi_{d-1})\in T^{-1}\circ G^{-1}_z(B\cap K)\right) \ge c^d \lambda_d\left(T^{-1}\circ G^{-1}_z(B\cap K)\right).
\end{equation}
We shall demonstrate below that
\begin{equation}
\label{ineq8} \lambda_d\left(T^{-1}\circ G^{-1}_z(B\cap K)\right) \ge \lambda_d(B\cap K).
\end{equation}
Indeed, for any injective continuously differentiable map $Q:K\to [0,1]^d$ and any measurable subset $A\subset K$,
\begin{equation}\label{transdet}
\lambda_d(Q(A))\ge \inf_{u\in A}\left|\det\left(\frac{\partial}{\partial u} Q(u)\right)\right|\cdot \lambda_d(A).
\end{equation}
We also observe that
\begin{align*}
\det\left(\frac{\partial}{\partial u} G_z^{-1}(u)\right)&=\det\begin{pmatrix}
1+\frac{z_1}{z_0} & \frac{z_1}{z_0} & \frac{z_1}{z_0}& \dots &\frac{z_1}{z_0}\\
\frac{z_2}{z_0} & 1+\frac{z_2}{z_0} & \frac{z_2}{z_0} & \dots &\frac{z_2}{z_0}\\
\vdots & \vdots & \ddots & \dots & \vdots \\
\frac{z_{d-1}}{z_0} & \dots & \frac{z_{d-1}}{z_0} & 1+\frac{z_{d-1}}{z_0} &\frac{z_{d-1}}{z_0}\\
\frac{z_d}{z_0} & \dots & \frac{z_d}{z_0} & \frac{z_d}{z_0} &1+\frac{z_d}{z_0}
\end{pmatrix}\\
&=\det\begin{pmatrix}
1+\frac{z_1}{z_0}& -1 & -1 & -1 & \dots &-1\\
\frac{z_2}{z_0} & 1 & 0 &0 & \dots &0\\
\frac{z_3}{z_0} & 0 & 1 &0 & \dots &0\\
\vdots & \vdots & \vdots & \ddots & \dots & \vdots \\
\frac{z_{d-1}}{z_0} & 0& \dots & 0 & 1 & 0\\
\frac{z_d}{z_0} & 0 & \dots & 0 & 0 &1
\end{pmatrix}
\\ &
=1+\frac{z_1}{z_0}+\frac{z_2}{z_0}+\dots+\frac{z_d}{z_0}=\frac{1}{z_0}\ge 1,
\end{align*}
where the second matrix is obtained from the first one by subtracting the first column from each of the remaining columns, and then we use {\em the Schur determinant identity}, i.e.\ $\det {\begin{pmatrix}C&D\\E&F\end{pmatrix}}=\det(F)\cdot \det(C-DF^{-1}E)$ when $F$ is invertible; here $F$ is the $d-1$ square identity matrix, $C=[1+z_1/z_0]$, etc.
Furthermore,
\begin{align*}
\det\left(\frac{\partial}{\partial v} T^{-1}(v)\right)&
=\det
\displaystyle
\begin{pmatrix} 
\frac{1}{1-\sum_{j=2}^d v_j} & \frac{v_1}{\left(1-\sum_{j=2}^d v_j\right)^2} & \frac{v_1}{\left(1-\sum_{j=2}^d v_j\right)^2} & \dots & \frac{v_1}{\left(1-\sum_{j=2}^d v_j\right)^2} \\
 0 & \frac{1}{1-\sum_{j=3}^d v_j} & \frac{v_2}{\left(1-\sum_{j=3}^d v_j\right)^2} & \dots & \frac{v_2}{\left(1-\sum_{j=3}^d v_j\right)^2} \\
 \vdots & \vdots & \ddots & \dots & \vdots \\
 0 & \dots & 0 & \frac{1}{1-v_d} & \frac{v_{d-1}}{(1-v_d)^2}\\
 0& \dots & 0 & 0 & 1
\end{pmatrix}
\\
&=\left[ \prod_{j=1}^d \left( 1- \sum_{l=j+1}^d v_l\right)\right]^{-1}\ge 1. 
\end{align*}
Therefore, the inequality~\eqref{ineq8} is obtained by applying~\eqref{transdet} to the map $Q=T^{-1}\circ G^{-1}_z$. 

Combining~\eqref{ineq}, \eqref{ineq7} and~\eqref{ineq8}, we conclude that for each $B\in \mathcal{B}(\mathcal{S}_d)$ and $z\in V_0$
\begin{align}\label{ineqV0}
\P\left(Z_d\in B\ |\ Z_0=z\right)\ge \text{Const}~ \lambda_d \left(B\cap K\right).
\end{align}
\noindent
\underline{\sf Step 4.} For each $k\in \{1,2,...,d\}$, $B\in \mathcal{B}(\mathcal{S}_d)$ and $z\in V_k$, we have 
\begin{align*}
& \P(Z_d\in B\ |\ Z_0=z)
 \\
& \ge \P\left(Z_d\in B\cap (K), (\Theta_{0},\dots,\Theta_{d-1})= (E_0,E_1,\dots,E_{k-1},E_{k+1},\dots, E_{d}) \ |\ Z_0=z\right)\\
&\ge \text{Const}~\P\left(R_k^{-1}\circ G_{R_k(z)}(\zeta_1,\zeta_2,...,\zeta_d) \in B\cap K \right)\\
&=\text{Const}~\P\left( (\zeta_1,\dots,\zeta_{d})\in 
 G_{R_k(z)}^{-1}\circ R_k(B\cap K ) \right)\\
&=\text{Const}~\P\left( (\xi_0,\xi_1,\dots,\xi_{d-1})\in 
T^{-1} \circ G_{R_k(z)}^{-1}\circ R_k(B\cap K ) \right),
\end{align*} 
where we use the fact that for $u\in \mathcal{S}_d$ and $z\in V_k$, 
$$R_k^{-1}(G_{R_k(z)}(u))=\left({u_0}{z_1} + {u_2}, \ldots ,{u_0}{z_{k - 1}} + {u_k},{u_0}{z_k},{u_0}{z_{k + 1}} + {u_{k + 1}},...,{u_0}{z_d} + {u_d}\right).$$
Similarly to the inequalities~\eqref{ineq7} and~\eqref{ineq8}, we have
\begin{align*}
\P\left( (\xi_0,\xi_1,\dots,\xi_{d-1})\in 
T^{-1} \circ G_{R_k(z)}^{-1}\circ R_k(B\cap K) \right)\ge \text{Const} ~\lambda_d(B\cap K).
\end{align*}
It follows that for each $B\in \mathcal{B}(\mathcal{S}_d)$, $k=1,2,...,d$ and $z\in V_k$,
\begin{align}\label{ineqVk} \P\left(Z_d\in B|Z_0=z\right)\ge \text{Const}~ \lambda_d \left(B\cap K \right).
\end{align}
Next, we define the probability measure $\phi$ as
\begin{align*}
\phi(B)=\frac{\lambda_d \left(B\cap K\right)}{\lambda_d \left(K \right)}.
\end{align*}
for each $B\in \mathcal{B}(\mathcal{S}_d)$.
From~\eqref{ineqV0} and~\eqref{ineqVk}, we can conclude that the condition (c) in Proposition~\ref{conderg} is verified. 
\\[2mm]
\underline{\sf Step 5.} For each $B\in \mathcal{B}(\mathcal{S}_d)$ and $z\in V$, 
\begin{align*} 
\P\left(Z_{d+1}\in B\ |\ Z_0=z\right)& \ge \P\left(Z_{d+1}\in B, Z_1\in V |Z_0=z\right)
\\& 
= \P\left(Z_{d+1}\in B\ |\ Z_1\in V,Z_0=z\right)\cdot \P(Z_1\in V\ |\ Z_0=z) \\ & 
\ge \eta~\P\left(Z_{d+1}\in B\ |\ Z_1\in V \right) 
\ge \text{Const}~ \lambda_d \left(B\cap K\right).
\end{align*}
Since $\gcd\{d,d+1\}=1$, the condition (d) in Proposition~\ref{conderg} is also fulfilled. 
\end{proof}

\section{Beta walks with linearly place-dependent probabilities}\label{sec3}
Suppose that the conditions of Assumption~\ref{asum1} are fulfilled. Since convergence in total variation implies convergence in distribution, as $n\to\infty$, $Z_n$ converges in distribution to a random vector $Z$ having the invariant measure. By the definition of the invariant measure
\begin{align}\label{eq_Zd}
Z\overset{d}{=} (1-\xi) Z + \xi \Theta,
\end{align}
where $\Theta$ is a discrete random vector satisfying 
\begin{align}\label{eqTh}
\P(\Theta= E_j\ |\ Z=z) = p_j(z),\qquad j=0,1,...,d,
\end{align} 
and $\xi$ is independent of~$Z$ and~$\Theta$.

\begin{lemma}\label{L31}
Assume that 
\begin{itemize}
\item $p=(p_1,p_2,\dots, p_d)$ is a  Borel measurable probability choice function and $\Theta$ satisfies~\eqref{eqTh};
\item $\xi$ is independent of~$Z$ and~$\Theta$;
\item $Z$ and $\xi$ have the probability density functions $f$ and $g$ respectively (w.r.t.\ Lebesgue measures $\lambda_d$ and $\lambda_1$).
\end{itemize}
Then \eqref{eq_Zd} holds if and only if $f$ and $g$ satisfy the following equation
\begin{equation}\label{density}
f(z)=\sum_{j=0}^d T_j(z) \qquad \lambda_d\text{--a.e.\ on }\mathcal{S}_d,
\end{equation}
where
\begin{align*}
T_0(z_1,z_2,\dots,z_d)&=\int_{z_1+z_2+\dots+z_d}^1\frac{1}{u^d}\, f\left(\frac{z_1}{u},\frac{z_2}{u},\dots,\frac{z_d}{u}\right)p_0\left( \frac{z_1}{u},\frac{z_2}{u},\dots,\frac{z_d}{u}\right)g(1-u)\rmd u,
\\
T_j(z_1,z_2,\dots,z_d)&=\int_{1-z_j}^1\frac{1}{u^d} f\left(\frac{z_1}{u},\dots \frac{z_{j-1}}{u},\frac{z_j-1+u}{u},\frac{z_{j+1}}{u},\dots\frac{z_d}{u}\right)\\
& \qquad \qquad \qquad \times p_j\left(\frac{z_1}{u},\dots \frac{z_{j-1}}{u},\frac{z_j-1+u}{u},\frac{z_{j+1}}{u},\dots\frac{z_d}{u}\right)g(1-u)\rmd u 
\end{align*}
for $j=1,2,\dots,d$. (The integrals above are understood in the Lebesgue sense.)
\end{lemma}
\begin{proof}
Denote $\tilde Z=(1-\xi) Z + \xi \Theta$. For each $z\in \mathcal{S}_d$, we have
\begin{align}\label{cdf1}\nonumber
\P\left(\tilde Z\le y\right)& =\int_0^1 \sum_{j=0}^d 
\P\left( u Z + ({1-u}) \Theta\le {y}, \Theta=E_j\right) g(1-u) \rmd u\\
& = \sum_{j=0}^d \int_{\mathcal{S}_d\times [0,1]}\mathbf{1}_{\left\{ z \le \frac{1}{u}(y-(1-u)E_j) \right\}}f(z)p_j(z)g(1-u) \rmd z\rmd u,
\end{align}
where for $y=(y_1,y_2,...,y_d),z=(z_1,z_2,...,z_d)\in \mathcal{S}_d$, we write $z\le y$ if $z_j\le y_j$ for all $j=1,2,...,d$.

For each $j\in \{0,1,\dots, d\}$, $u\in(0,1)$ and $y\in \mathcal{S}_d^o$, changing the variable $x=\phi(z):=uz+(1-u)E_j$ 
we have
\begin{align}\label{cdf3}
\nonumber &\int_{\mathcal{S}_d}\mathbf{1}_{\left\{z \le\frac{1}{u}(y-(1-u)E_j) \right\}}f(z)p_j(z)\rmd z=\int_{\phi(\mathcal{S}_d)}\mathbf{1}_{\left\{ x \le y \right\}}f(\phi^{-1}(x))p_j(\phi^{-1}(x))
{\sf D} \phi^{-1}(x)
 \rmd x\\
%
&=\int_{\left\{x\in \mathcal{S}_d\ :\ 0 \le x\le y, 1-u\le x_j\le 1 
\right\}}\frac{1}{u^d}f\left( \frac{1}{u}(x-(1-u)E_j) \right)p_j\left( \frac{1}{u}(x-(1-u)E_j) \right) \rmd x,
\end{align}
where ${\sf D}\phi^{-1}$ denotes the Jacobian of $\phi^{-1}$.
Combining~\eqref{cdf1} and~\eqref{cdf3}, and applying Fubini's theorem, we obtain
\begin{align*}&\P\left(\tilde Z\le y\right)\\
&=
\int_{\{x\in \mathcal{S}_d\ :\ 0\le x\le  y \}}\left[\sum_{j=0}^d\int_{1-x_j}^1 \right.
 \left. \frac{1}{u^d}f\left( \frac{1}{u}(x-(1-u)E_j) \right)p_j\left( \frac{1}{u}(x-(1-u)E_j) \right) g(1-u) \rmd u\right] \rmd x.
\end{align*}
Therefore,
\begin{align*}
\tilde{f}(z): =\sum_{j=0}^d\int_{1-z_j}^1 \frac{1}{u^d} f\left( \frac{1}{u}(z-(1-u)E_j) \right)p\left( \frac{1}{u}(z-(1-u)E_j) \right) g(1-u) \rmd u = \sum_{j=0}^d T_j(z)
\end{align*}
is a probability density function of $\tilde Z$, which is unique up to a set of measure zero. Hence, $f(z)=\tilde f(z)$ for almost all $z\in \mathcal{S}_d$, and the lemma is thus proved. 
\end{proof}

\begin{theorem}\label{linearmodel}Assume that 
\begin{enumerate}
\item[(a)] $\xi\sim \text{\rm Beta}(1,\gamma)$, where $\gamma>0$ is some constant;
\item[(b)] $p=(p_1,p_2,\dots,p_d): \mathcal{S}_d\to \mathcal{S}_d$ is defined by
$$
p_k(z_1,z_2, \dots,z_d)= \displaystyle
\beta_k(1-z_k) +\left( 1-\sum_{j=1}^{d+1}\beta_j+\beta_k \right) z_k, \qquad k=1,2,...,d,
$$
where $\beta_k>0$ and $\sum_{j=1}^{d+1}\beta_j-\beta_k<1$ for $k=1,2,...,d+1$;
\item[(c)] $Z\sim {\rm Dirichlet}(\beta_1\gamma,\beta_2\gamma,\dots,\beta_d\gamma, \beta_{d+1}\gamma)$.
\end{enumerate}
Then $Z\overset{d}{=} (1-\xi) Z + \xi \Theta$, and thus $Z_n$ converges to a Dirichlet distribution by Theorem~\ref{t1} and Lemma~\ref{L31}.
\end{theorem}
\begin{proof} Let $f$ and $g$ be respectively the probability density functions of $\text{Dirichlet}(\beta_1\gamma,\dots,\beta_d\gamma, \beta_{d+1}\gamma)$ and $\text{Beta}(1,\gamma)$. It suffices to check that $f$ and $g$ satisfy the integral equation~\eqref{density}.

We have
\begin{align*}
T_0(z_1,\dots,z_d) &= \frac{\displaystyle\Gamma \left( \gamma \sum\limits_{j = 1}^{d + 1} \beta_j \right)}{\displaystyle\prod\limits_{j = 1}^{d + 1} \Gamma (\beta_j\gamma )}\prod\limits_{j = 1}^d z_j^{\beta _j\gamma - 1}\int_{\sum\limits_{j = 1}^d z_j}^1 \left( u - \sum\limits_{j = 1}^d z_j \right)^{\gamma\beta_{d + 1} - 1} u^{-\gamma \left( \sum\limits_{j = 1}^{d + 1} \beta_j - 1 \right)-1}\\
 & \times\left[\gamma\beta_{d+1} u -\gamma\left(\sum_{j=1}^{d+1}\beta_j-1\right)\left(u-\sum\limits_{j = 1}^d z_j\right)\right]\rmd u
\\
 &=\frac{\displaystyle\Gamma \left( \gamma \sum\limits_{j = 1}^{d + 1} \beta_j \right)}{\displaystyle\prod\limits_{j = 1}^{d + 1} \Gamma (\beta_j\gamma )}
 \left(1-\sum_{j=1}^d z_j \right)^{\beta_{d+1}\gamma}\prod_{j=1}^d z_j^{\beta_j\gamma-1},
\end{align*}
where we use the fact that 
$$
\int_z^1 u^{-b-1} (u-z)^{a-1} [a u-b (u-z)] \rmd u=(1-z)^a.
$$
Similarly, for $k=1,2,...,d$, we also obtain that
$$T_k(z_1,\dots,z_d)=\frac{\displaystyle\Gamma\left( \gamma\sum\limits_{j=1}^{d+1}\beta_j\right)}{\displaystyle\prod\limits_{j=1}^{d+1}\Gamma(\beta_j \gamma)}\left(1-\sum_{j=1}^d z_j \right)^{\beta_{d+1}\gamma-1}z_k\prod_{j=1}^d z_j^{\beta_j\gamma-1}.$$
Therefore,
$$\sum_{k=0}^dT_k(z_1,\dots,z_d)=\frac{\displaystyle\Gamma\left( \gamma\sum\limits_{j=1}^{d+1}\beta_j\right)}{\displaystyle\prod\limits_{j=1}^{d+1}\Gamma(\beta_j \gamma)}\left(1-\sum_{j=1}^d z_j \right)^{\beta_{d+1}\gamma-1}\prod_{j=1}^d z_j^{\beta_j\gamma-1}=f(z_1,z_2,...,z_d).$$
\end{proof}

We want to conclude this Section with the following observations.
\begin{itemize}
\item If $d=1$, then $p_1(z)=\beta_1(1-z)+(1-\beta_2)z$. This is, in fact, the one-dimensional case considered by McKinlay and Borovkov in~\cite{McKinlay}. 
\item
If $d\ge 1$ and $\sum_{j=1}^{d+1}\beta_j=1$, then we obtain the model considered by Sethuraman in~\cite{Sethuraman}.
\end{itemize} 

\section{Random walks in [0,1] based on urn-type schemes}
Let $h:[0,1]\to\R_+$ be some non-random measurable function. In this section, we will study a random walk on the unit interval $\mathcal{S}_1=[0,1]$ with the following properties: 
\begin{enumerate}
\item At time $n=0,1,2,\dots$, the system is characterized by $Z_n\in [0,1]$ (location of the particle) and two positive numbers $L_n$ and $R_n$. We assume that $R_0=L_0=1$. 

\item At time $n+1$, with probability $\frac{L_n}{L_n+R_n}$ the quantity $L_n$ increases by $h(Z_n)$, i.e.\ a function of the distance from $0$ to the current position of the particle, and then the particle jumps to a new location $Z_{n+1}$, which is uniformly distributed on the interval $(0,Z_n)$. With the complementary probability $\frac{R_n}{L_n+R_n}$, the quantity $R_n$ increases by $h(1-Z_n)$, i.e.\ a function of the distance from~$1$ to the current position of the particle, and then the particle jumps to a new location $Z_{n+1}$, uniformly distributed on the interval $(Z_n,1)$.
\end{enumerate}
One can think of $L_n$ and $R_n$ as numbers of two different kinds of balls in an urn, and the direction of the walk is governed by the kind of ball that is drawn randomly from the urn at time $n$. The number of balls of the chosen type then increases by yet another random quantity, depending on the position of the walk. The number of balls in our model can be, in general, non-integer; this is, however, allowed for the generalized P\'olya urn models.

Formally, we can write the model as the following recursion: let $Z_0\in(0,1)$ be some non-random quantity, and for $n=1,2,\dots$ let
\begin{equation}\label{recur}
\begin{array}{ll}
Z_{n}&= Z_{n-1} \xi_n\cdot \mathbf{1}_{\left\{U_{n}<\frac{L_{n-1}}{L_{n-1}+R_{n-1}}\right\}}
+
\left[Z_{n-1}+(1-Z_{n-1})\,\xi_n\right]\cdot \mathbf{1}_{\left\{U_n\ge \frac{L_{n-1}}{L_{n-1}+R_{n-1}}\right\}},
\\
L_n&= L_{n-1}\ \ \ +h(Z_{n-1})\cdot \mathbf{1}_{\left\{U_n<\frac{L_{n-1}}{L_{n-1}+R_{n-1}}\right\}},\\
R_n&= R_{n-1}+h\left(1-Z_{n-1}\right)\cdot \mathbf{1}_{\left\{U_n\ge \frac{L_{n-1}}{L_{n-1}+R_{n-1}}\right\}},
\end{array}
\end{equation}
where $\left\{\xi_n,U_n\right\}_{n=1}^\infty$ is a set of i.i.d.\ uniform $U[0,1]$ random variables. 
Since the probabilities of jumps to the left (and right resp.) depend on $(L_n, R_n)$, the distribution of $Z_{n+1}$ is generally dependent on the whole history of the random walk up to time $n$. Let also 
\begin{align*}
\mathcal{F}_n&=\sigma(U_1,...,U_n,\xi_1,\dots,\xi_n),
\\
\mathcal{G}_n&=\sigma(U_1,...,U_n,\xi_1,\dots\xi_{n-1}),
\end{align*}
and note that $Z_n$ is $\mathcal{F}_n$-measurable, while $L_n$ and $R_n$ are $\mathcal{G}_n$-measurable. 

If $h(x)\equiv 0$, then $L_n$ and $R_n$ do not change with time, and $Z_n$ is a Markov chain satisfying Theorem \ref{linearmodel} with $\gamma=1$ and $p\equiv p_1\equiv \frac{R_0}{L_0+R_0}$, thus $Z_n$ converges in distribution to Beta$\left(\frac{R_0}{L_0+R_0},\frac{L_0}{L_0+R_0}\right)$. If $h(x)\equiv \beta$ for some constant $\beta>0$, then the process $(L_n,R_n)$ is the classical P\'olya urn. We conjecture that under some regularity conditions on the function $h$, the random walk $Z_n$ converges either almost surely to a Bernoulli random variable, or weakly to some non-trivial distribution with full support on $[0,1]$ (compare with Section~2.1 in~\cite{Diaconis}). Even though we were not able to deal with the general case, there is one non-trivial situation, where we have explicit results, as follows. 

In the remaining part of this Section, we consider only the case where 
$$
h(x)=x,\qquad x\in [0,1].
$$
It turns out that even in this seemingly ``simple" case, there are challenges to rigorously obtain the limiting distribution (see Theorem~\ref{t3} below). 

\begin{lemma}\label{sumLR}
We have
\begin{itemize}
\item[(a)]
$\limsup_{n\to\infty} \frac{L_{n}}{R_{n}}\le 4$ and
$\limsup_{n\to\infty} \frac{R_{n}}{L_{n}}\le 4$ almost surely;
\item[(b)] $L_n\to\infty$ and $R_n\to\infty$ almost surely as $n\to\infty$. 
\end{itemize}
\end{lemma}
\begin{proof}
First of all, observe that the probability that the sequence $Z_n$ eventually becomes monotone is zero, namely 
$$
\P(\exists N:\ Z_{n+1}\le Z_n\ \forall n\ge N)=0
\quad\text{and}\quad 
\P(\exists N:\ Z_{n+1}\ge Z_n\ \forall n\ge N)=0. 
$$
Since $R_n$ is non-decreasing in $n$ and $L_n\le L_0+n$, we have
$$
\P\left(Z_{n+1}\in(Z_n,1)\ |\ \mathcal{F}_n\right)
=\frac{R_n}{L_n+R_n}\ge 
\frac{R_0}{L_n+R_0}\ge 
\frac{R_{0}}{R_0 +L_0+n}.
$$
Since $\displaystyle\sum_{n=1}^{\infty} \frac{R_0}{R_0 +L_0+n}=\infty$,
by Levy's extension to the Borel-Cantelli lemma the event in the above display happens infinitely often with probability $1$, hence there are infinitely many $n$s for which $Z_n$ decreases. By the identical argument, $Z_n$ cannot become eventually increasing.
\\[5mm]
Let us prove part (a) now.
We know that $Z_n$ makes a.s.\ infinitely many steps to the left as well as to the right. Hence there exists a sequence of finite stopping times with respect to the filtration $\mathcal{G}$: 
\begin{align*}
\tau_1&=0,\\
\eta_i&=\inf\{n>\tau_i:\ Z_n\ge Z_{n-1}\},\\
\tau_{i+1}&=\inf\{n>\eta_{i}:\ Z_n< Z_{n-1}\}
\end{align*}
for $i=1,2,\dots$.
Moreover,
$$
\tau_1<\eta_1<\tau_2<\eta_2<\dots,
$$ 
$\tau_n,\eta_n\to\infty$ as $n\to\infty$, and
\begin{align*}
Z_{n}<Z_{n-1},&\ \text{ if } n\in[\tau_i,\eta_i-1],\\
Z_{n}\ge Z_{n-1},&\ \text{ if } n\in[\eta_i,\tau_{i+1}-1]
\end{align*}
for each $i=2,3,\dots$ (note that, in fact, the probability of the event $Z_n=Z_{n-1}$ is zero).

Observe that for each $k\ge 1$ and $1\le \ell\le k-1$,
\begin{align*}
& \left\{\eta_i=k,\ \tau_i=k-\ell\right\}
=
\left\{
Z_k\ge Z_{k-1},\ Z_{k-1}< Z_{k-2}<\dots<Z_{k-\ell}, \tau_i=k-\ell
\right\}
\\ &
=\left\{
U_k\ge \frac{ L_{k-1}}{L_{k-1}+R_{k-1}},\ 
U_{k-1}<\frac{ L_{k-2}}{L_{k-2}+R_{k-2}}, \dots,
U_{k-\ell+1}<\frac{ L_{k-\ell}}{L_{k-\ell}+R_{k-\ell}}
,\ \tau_i=k-\ell\right\}.
\end{align*}
However, $L_{k-j}$ and $R_{k-j}$ depend only on $L_{k-j-1}$, $R_{k-j-1}$, $Z_{k-j-1}$ and $U_{k-j}$, $j=1,2,\dots,\ell$ (see~\eqref{recur}). Hence the event above is $\sigma(U_1,\dots,U_k,\xi_1,\dots,\xi_{k-2})$--measurable, and it is thus independent of $\xi_{k-1}$; as a result $V_{i+1}:=\xi_{\eta_i-1}$ is independent of 
$$
\mathcal{H}_i:=\sigma\left(U_1,U_2,\dots,U_{\eta_i},\xi_1,\xi_2,\dots,\xi_{\eta_i-2}\right).
$$
On the other hand, since $\eta_{i-1}-1\le \eta_i-2$, $V_i$ is $\mathcal{H}_i$-measurable. So $\{V_i\}_{i=1}^\infty$ is an i.i.d.\ sequence of Uniform$[0,1]$ random variables.

We have
$$
R_{\eta_i}-R_{\eta_i-1}=1-Z_{\eta_i-1}=1-Z_{\eta_i-2} \xi_{\eta_i-1}\ge 1-\xi_{\eta_i-1}=1-V_{i+1},
$$
hence, due to the monotonicity of $R_n$,
$$
R_{\eta_i}\ge \sum_{k=1}^{i} [1-V_{k+1}].
$$
By the strong law of large numbers $\displaystyle\lim_{k\to\infty} \frac 1k \sum_{i=1}^k [1-V_{i+1}]=\frac 12$ a.s., hence 
$$
\liminf_{k\to\infty} \frac{R_{\eta_k}}{\eta_k}\ge \frac 12\qquad\text{a.s.}
$$

Next, for $i\in\N$, we have
\begin{align*}
L_{\tau_{i+1}-1}-L_{\eta_i-1}&=0,\\
L_{\eta_i-1}-L_{\tau_i-1}&=Z_{\tau_i-1}(1+\xi_{\tau_i}+\xi_{\tau_i}\xi_{\tau_i+1}+\dots+\xi_{\tau_i}\xi_{\tau_i+1}\cdots\xi_{\eta_i-2})
\\ &\le 1+\xi_{\tau_i}+\xi_{\tau_i}\xi_{\tau_i+1}+\dots+\xi_{\tau_i}\xi_{\tau_i+1}\cdots\xi_{\eta_i-2}\le \tilde V_i,
\end{align*}
where
\begin{align*}
\tilde V_i&=\left[1+\xi_{\tau_i}+\xi_{\tau_i}\xi_{\tau_i+1}+\dots+\xi_{\tau_i}\xi_{\tau_i+1}\cdots\xi_{\eta_i-2}\right]
\\ &
+\xi_{\tau_i}\xi_{\tau_i+1}\cdots\xi_{\eta_i-2}
\tilde\xi^{(i)}_{\eta_i-1}
+\xi_{\tau_i}\xi_{\tau_i+1}\cdots\xi_{\eta_i-2}
\tilde\xi^{(i)}_{\eta_i-1}\tilde\xi^{(i)}_{\eta_i}
+\xi_{\tau_i}\xi_{\tau_i+1}\cdots\xi_{\eta_i-2}
\tilde\xi^{(i)}_{\eta_i-1}\tilde\xi^{(i)}_{\eta_i}\tilde\xi^{(i)}_{\eta_i+1}+\dots
\\
&=1+
\sum_{k=\tau_i}^{\eta_i-2} \xi_{\tau_i}\xi_{\tau_i+1}\cdots\xi_{k}
+
\sum_{k=\eta_i-1}^{\infty} 
\xi_{\tau_i}\cdots\xi_{\eta_i-2}\cdot
\tilde\xi^{(i)}_{\eta_i-1}\tilde\xi^{(i)}_{\eta_i}
\cdots\tilde\xi^{(i)}_{k}
\end{align*}
and $\tilde\xi^{(i)}_k$, $i,k\in\N$ are i.i.d.\ copies of $\xi$, independent of everything else. By construction, $\tilde V_i$ is a $\tilde{\mathcal{H}_i}$-measurable random variable, where 
$$
\tilde{\mathcal{H}_i}:=\sigma\left(U_1,U_2,\dots,U_{\eta_i};\,\xi_1,\xi_2,\dots,\xi_{\eta_i-2};\, 
\tilde\xi^{(\ell)}_k, \ell=1,2,\dots, i,\ k\in\N \right).
$$ 
On the other hand, one can easily show that the variables $\xi_{\tau_{i+1}+j}$, $\tilde{\xi}^{(i+1)}_{j+1}$, $j=0,1,2,\dots$, are independent of $\tilde{\mathcal{H}_i}$, and therefore $\tilde V_{i+1}$ is independent of $\tilde{\mathcal{H}_i}$. Consequently, $\tilde V_i$, $i=1,2,\dots$, are independent random variables with expectation $\displaystyle\E\left[1+ \sum_{k=1}^{\infty} \left(\prod_{i=1}^k \xi_i\right)\right]=2$, hence by the strong law we have $\displaystyle\lim_{k\to\infty} \frac 1k \sum_{i=1}^k \tilde V_i=2$ a.s., yielding
\begin{align}\label{eqlimsup}
\limsup_{k\to\infty} \frac{L_{\eta_k}}{\eta_k}\le 2\qquad\text{a.s.}
\end{align}

Combining this with~\eqref{eqlimsup}, and taking into account that $\eta_i\to\infty$, we get
$$
\limsup_{n\to\infty} \frac{L_n}{R_n}\le \frac{2}{1/2}={4}
\qquad\text{a.s.}
$$
Due to the symmetry, the complementary inequality can be proved identically.
\\[5mm]
Let us now prove part (b).
From part (a) we obtain that a.s.\ either both $L_n$ and $R_n$ increase to $\infty$, or both stay bounded, i.e.\ $\sup_{n\ge0}L_n<\infty$ and $\sup_{n\ge0}R_n<\infty$ for all $n$. Let us show that the latter case a.s.\ cannot happen. Again, from (a) we get that a.s.\ there exists a (random) $N$ such that for $n\ge N$
$$
\P(Z_{n+1}<Z_n| {\cal F}_n)=\frac{L_n}{L_n+R_n}
=\frac{1}{1+R_n/L_n}>\frac1{5} , \quad
\P(Z_{n+1}>Z_n| {\cal F}_n)=\frac{R_n}{L_n+R_n}>\frac 1{5} .
$$
As a result, for $n\ge N$, we have for $S_n=L_n+R_n$
\begin{align*}
& \P\left(S_{n+1}-S_n>\frac 14\ |\ {\cal F}_n\right)\\
& =
\P\left(S_{n+1}-S_n>\frac 14\ |\ {\cal F}_n, Z_n\le \frac12\right)
+
\P\left(S_{n+1}-S_n>\frac 14\ |\ {\cal F}_n, Z_n> \frac12\right)
\\
&\geq
\P\left(R_{n+1}-R_n>\frac 14\ |\ {\cal F}_n, Z_n\le \frac12\right)
+
\P\left(L_{n+1}-L_n>\frac 14\ |\ {\cal F}_n, Z_n> \frac12\right)
\\
&=
\P\left(R_{n+1}-R_n>\frac 14,\ Z_{n+1}>Z_n\ |\ {\cal F}_n, Z_n\le \frac12\right)
\\ &+
\P\left(L_{n+1}-L_n>\frac 14,\ Z_{n+1}<Z_n\ |\ {\cal F}_n, Z_n> \frac12\right)
\\
&\geq \frac 12 \cdot\frac 15+\frac 12 \cdot\frac 15=\frac 15>0.
\end{align*}
Since $S_n$ is non-decreasing for any $n$, this implies that $S_n\to\infty$ a.s., contradicting the assumption that both $L_n$ and $R_n$ remain bounded.
\end{proof}
\begin{lemma}\label{limitzeta}
$\zeta_n:=\displaystyle \frac{L_{n}}{L_{n}+R_{n}}$ converges almost surely to $\zeta_{\infty}\in(0,1)$ as $n\to\infty$.
\end{lemma}
\noindent
\textbf{Remark.}
Lemma~\ref{sumLR} implies only that
$\displaystyle \frac 1{5}\le \liminf_{n\to\infty}\zeta_n\le \limsup_{n\to\infty} \zeta_n\le \frac{4}{5}$ a.s.
\vskip 5mm
\noindent
\begin{proof}[Proof of Lemma~\ref{limitzeta}]
We introduce the quantity 
$$
\displaystyle W_n=\left(\frac{1}{2}-\frac{L_n+Z_{n}}{L_{n}+R_n}\right)^2+\frac 1{L_n+R_n}
$$
which will be shown to be a supermartingale.
Indeed, 
\begin{align*}
& \E\left(W_{n+1}|\mathcal{F}_n\right)\\
&=\frac{L_n}{L_n+R_n}\E\left(W_{n+1}\left|U_{n+1}<\frac{L_{n}}{L_n+R_n},\mathcal{F}_n\right.\right)\\ 
&+\frac{R_n}{L_n+R_n}\E\left(W_{n+1}\left|U_{n+1}>\frac{L_n}{L_n+R_n},\mathcal{F}_n\right.\right)\\
& = \frac{L_n}{L_n+R_n}\int_{0}^{Z_n} \frac{1}{Z_n}\left[ \left(\frac{1}{2}-\frac{L_n+Z_n+u}{L_n+R_n+Z_n} \right)^2+\frac{1}{L_n+R_n+Z_n}\right]\rmd u\\
& + \frac{R_{n}}{L_n+R_{n}}\int_{Z_n}^{1} \frac{1}{1-Z_n}\left[\left(\frac{1}{2}-\frac{L_n+u}{L_n+R_n+1-Z_n} \right)^2+\frac{1}{L_n+R_n+1-Z_n}\right]\rmd u
\\
&= \frac{L_n}{L_n + R_n}\frac{3\left( {L_n - R_n} \right)^2 + 12\left( {L_n - R_n} \right)Z_n + 12\left( {L_n + R_n + Z_n} \right) + 13Z_n^2}{12\left( L_n + R_n + Z_n \right)^2}\\
&+ \frac{R_n}{L_n + R_n}\frac{3\left( {L_n - R_n} \right)^2 + 12\left( {L_n - R_n} \right)Z_n + 12\left( {L_n + R_n + Z_n} \right) + 13(1-Z_n)^2}{12\left( L_n + R_n + 1-Z_n \right)^2}.
\end{align*}
Substituting $\displaystyle\zeta_n=\frac{L_n}{L_n+R_n}$ and $ \displaystyle \eps_n=\frac{1}{L_n+R_n}$, or, equivalently, $\displaystyle L_n= \frac{\zeta_{n}}{\eps_n}, R_n=\frac{1-\zeta_n}{\eps_n}$, we obtain that
\begin{equation}\label{supermartingale}
\E\left(W_{n+1} -W_n |\mathcal{F}_n,Z_n=z\right)=
\frac{\eps_n\left[ r_0(\zeta_n,z)+ r_1(\zeta_n,z) \eps_n+\dots +r_5(\zeta_n,z) \eps_n^5\right]}{
6(\eps_n z+1)^2 (1+\eps_n(1-z))^2},
\end{equation}
where
\begin{align*}
r_0(\zeta,z)&=-24z \zeta^3+36z \zeta^2+12\,\zeta^3-18z\zeta-24\zeta^2+3z+15\zeta-3
\\
&=-3(2\zeta-1)^2(\zeta z+(1-z)(1-\zeta)),
\\ r_1(\zeta,z)&=-30\,{z}^{2}{\zeta}^{2}-12\,z{\zeta}^{3}+30\,{z}^{2}\zeta+24\,z{\zeta}
^{2}+6\,{\zeta}^{3}-7\,{z}^{2}-38\,z\zeta-12\,{\zeta}^{2}+14\,z+13\,
\zeta-7,
\\ r_2(\zeta,z)&=10\,{z}^{3}\zeta-12\,{z}^{2}{\zeta}^{2}-5\,{z}^{3}-9\,{z}^{2}\zeta+7\,
{z}^{2}-19\,z\zeta+4\,z+6\,\zeta-6,
\\ r_3(\zeta,z)&=-z \left( 6\,{z}^{3}{\zeta}^{2}-6\,{z}^{3}\zeta-12\,{z}^{2}{\zeta}^{2}
-17\,{z}^{3}-12\,{z}^{2}\zeta+6\,z{\zeta}^{2}+28\,{z}^{2}+30\,z\zeta-
23\,z-6\,\zeta+12 \right), 
\\ 
r_4(\zeta,z)&=-6\,{z}^{2} \left( 1-z \right) \left(z^2+2z\zeta(1-z)+1 \right), 
\\ r_5(\zeta,z)&=-6z^4 \left( 1- z \right)^2.
\end{align*}
One can show that $\max_{x,y\in [0,1]} r_i(x,y)\le 0$ for $i=0,2,3,4,5$. From the assertion of Lemma~\ref{sumLR}, $\eps_n\to 0$ almost surely as $n\to\infty$, and one can show that $\max_{x,y\in [0,1]} r_0(x,y)+r_1(x,y)\eps \le 0$ for $0\le \eps< 0.5$. Hence $W_n$ is a supermartingale. Therefore, by Doob's martingale convergence theorem, a.s.\ there exists $\displaystyle W_{\infty}:=\lim_{n\to\infty} W_n$. Observe that $$\displaystyle\zeta_n\in \left\{\frac12-\sqrt{W_n-\frac1{L_n+R_n}}-\frac{Z_n}{L_n+R_n},\frac12+\sqrt{W_n-\frac1{L_n+R_n}}-\frac{Z_n}{L_n+R_n}\right\}.$$ On the other hand, note that 
$$
\left|\zeta_{n+1}-\zeta_n\right|\le \max \left\{\frac{L_n}{L_n+R_n}.\frac{1-Z_n}{L_{n+1}+R_{n+1}},\frac{R_n}{L_n+R_n}.\frac{Z_n}{L_{n+1}+R_{n+1}} \right\}\to 0
$$
as $n\to\infty$. As a result, $\limsup_{n\to\infty} \zeta_n=\liminf_{n\to\infty} \zeta_n=:\zeta_{\infty}$ almost surely.
\end{proof}

\begin{lemma}\label{CVas} $\displaystyle\zeta_{\infty}=\frac{1}{2}$ almost surely. 
\end{lemma}
\begin{proof}
Suppose $\P(\zeta_{\infty}=1/2)<1$. Then there exists $\eps>0$ such that $\P(|\zeta_{\infty}-1/2|>\eps)>0$. Let us denote the stopping time 
$$
\tau_{m}=\inf\left\{n\ge m \ : \ \left|\zeta_{n}-\frac12\right|<\frac{\eps}{2}\right\}.
$$
Since $\P(|\zeta_{\infty}-1/2|>\eps)>0$, there exists $m$ such that $\P(\tau_m=\infty)>0$. Let us consider $Y_{n}=W_{n\wedge \tau_m}$. $Y_n$ is also a supermartingale, hence, there exists $Y_{\infty}=\lim_{n\to\infty}Y_n$ as well. From~\eqref{supermartingale} it follows
\begin{equation*}
\E\left(W_{n+1} -W_n |\mathcal{F}_n\right)=\frac{1}{6}\eps_n \left(r_0(\zeta_n,Z_n)+ 
\eps_n \rho_n\right),
\end{equation*}
where 
\begin{align*}
&\rho_n =\frac{1}{(1+\eps_n Z_n)^2 (1+\eps_n(1-Z_n))^2}\left( {r}_1(\zeta_n,Z_n)+{r_2}(\zeta_n,Z_n)\eps_n+\dots +r_5(\zeta_n,Z_n) \eps_n^4
\right.
\\
& \left.- \left(1+(1-Z_n)Z_n\eps_n \right)\left(2+\eps_n+(1-Z_n)Z_n\eps_n^2\right)r_0(\zeta_n,Z_n) 
\right)
\end{align*}
Furthermore, $|\rho_n|$ is bounded by a non-random constant. This fact implies that for $N>0$
$$
\E\left( W_{m+N} \right)-\E\left( W_{m} \right)=\frac{1}{6}\E\left(\sum_{n=m}^{m+N} \eps_n\left( r_0(\zeta_n,Z_n)+O(\eps_n)\right)\right).
$$
Therefore,
\begin{equation}\label{sumY}
\E\left( Y_{\infty} \right)-\E\left( Y_{m} \right)=\frac{1}{6}\E\left(\sum_{n=m}^{\tau_m} \eps_n(r_0(\zeta_n,Z_n)+O(\eps_n))\right).
\end{equation}
Note that $\left|\zeta_{n}-\frac12\right|\ge \frac{\eps}{2}$, for all $n\in[m,\tau_m)$. Hence, combining with the remark that we made immediately after the statement of Lemma~\ref{limitzeta}, it follows that on the event $\{\tau_m=\infty\}$
$$
r_0(\zeta_n,Z_n)=-3\left(2 \zeta_n-1\right)^2 \left( Z_n \zeta_n+(1-Z_n)(1-\zeta_n)\right)
 \le- 3{\eps^2}\min\{\zeta_n,1-\zeta_n\}\le -\frac{3\eps^2}{13}
$$
for large enough $n$. Since $\P(\tau_m=\infty)>0$ and $\eps_n\ge \frac{1}{2+n}$, the LHS of~\eqref{sumY} is finite while the RHS is divergent. This contradiction proves the lemma. 
\end{proof}

\begin{theorem}\label{t3}
As $n\to\infty$, $Z_n$ converges in distribution to a $\text{\rm Beta}\left(\frac12,\frac12\right)$ random variable.
\end{theorem}
\begin{proof}
Let us fix a small $\eps>0$. By Lemma~\ref{CVas} there exists a (random) $N$ such that 
$$
\displaystyle\frac{1}{2}-\eps\le\frac{L_{n}}{L_n+R_n}\le \frac{1}{2}+\eps
$$ 
for all $n\ge N$. Fix a large non-random $N_0$. For this fixed $N_0$ we couple $\{{Z_n}\}_{n\ge 0}$ with two random walks $\{\tilde{Z_n}\}_{n\ge 0}$ and $\{\hat Z_n\}_{n\ge 0}$ defined as follows: 
\begin{itemize}
\item For $0\le n \le N_0$, set $\tilde{Z_n} = \hat Z_n =Z_{n}.$
\item For $n\ge N_0$, set
$$
\tilde Z_{n+1}=\begin{cases}
\xi_{n+1} \tilde Z_n & \text{if} \ U_{n+1}\le \frac{1}{2}-\eps;\\
\tilde Z_n+\xi_{n+1} (1-\tilde Z_n) & \text{if} \ U_{n+1}>\frac{1}{2}-\eps,
\end{cases}
$$
and
$$
\hat Z_{n+1}=\left\lbrace\begin{matrix}
 \xi_{n+1} \hat Z_n & \text{if} \ U_{n+1}\le \frac{1}{2}+\eps;\\
\hat Z_n+\xi_{n+1} (1-\hat Z_n) & \text{if} \ U_{n+1}> \frac{1}{2}+\eps.
\end{matrix}\right.
$$ 
\end{itemize}
Let $A_{N_0}=\{N\le N_0\}$.

Assume that for some $n\ge N_0$, $\hat Z_n\le Z_n\le \tilde Z_{n}$ (this is definitely true for $n=N_0$). We observe that on $A_{N_0}$: 
\begin{itemize}
\item 
when $\tilde Z_n$ chooses left, $Z_n$ also chooses left since $U_{n+1}\le \frac12-\eps<\frac{L_n}{L_n+R_n}$. In this case, $Z_{n+1}=\xi_{n+1}Z_n\le \xi_{n+1}\tilde Z_n = \tilde Z_{n+1}$. When $\tilde Z_n$ chooses right, $Z_n$ might choose left or right, but we still have $Z_{n+1}\le Z_n+\xi_{n+1}(1-Z_n)\le \tilde Z_n+\xi_{n+1}(1-\tilde Z_n) = \tilde Z_{n+1}$; 
\item 
when $Z_n$ chooses left, $\hat Z_n$ also chooses left since $U_{n+1}\le \frac{L_n}{L_n+R_n} < \frac12+\eps$. In this case, $Z_{n+1}=\xi_{n+1}Z_n\ge \xi_{n+1}\hat Z_n = \hat Z_{n+1}$. When $Z_n$ chooses right, $\hat Z_n$ might choose left or right, but we still have $\hat Z_{n+1}\le \hat Z_n+\xi_{n+1}(1-\hat Z_n)\le Z_n+\xi_{n+1}(1- Z_n) = Z_{n+1}$. 
\end{itemize}
By induction, we obtain that on $A_{N_0}$ for all $n\ge 0$, $\hat Z_{n}\le Z_n\le \tilde Z_{n}$. Therefore, we have
$$
\P(\tilde Z_n\le x,\ A_{N_0})\le \P(Z_n\le x,\ A_{N_0}) \le \P(\hat Z_n\le x,\ A_{N_0})
$$ 
for all $n\ge0$ and $x\in [0,1]$. On the other hand, by Theorem~\ref{linearmodel}, $\tilde Z_{n}$ and $\hat Z_{n}$ converge weakly to $\text{Beta}\left(\frac{1}{2}+\eps,\frac{1}{2}-\eps\right)$ and $\text{Beta}\left(\frac{1}{2}-\eps,\frac{1}{2}+\eps\right)$ respectively, as $n\to\infty$. Since $\eps$ is arbitrarily small and $\P(A_{N_0})\to 1$ as $N_0\to\infty$, the theorem is proved.
\end{proof}

\section*{Appendix}
\noindent 
\textbf{Proof of Lemma~\ref{transform}}\\
(a) For $u=(u_1,...,u_d)\in K$ and $z=(z_1,...,z_d)\in V_0$ define
$$
v=(v_1,v_2,\dots, v_d):= G_z^{-1}(u).
$$ 
Note that $u_0\le 1-u_d\le 1-\delta \le z_0$, $z_j\le 1-z_0\le \delta$, thus 
$$
s(1-t)^{d-1} -\delta\le u_j - z_j \le v_j=u_j-\frac{u_0}{z_0}z_j\le u_j\le t
$$ 
for $j=1,2,...,d$.Therefore, for $j=1,2,\dots,d$, we have
$$
s(1-t)^{d-1}-\delta \le v_j\le \frac{v_j}{\displaystyle 1-\sum_{l=j+1}^d v_l}=\frac{\displaystyle u_j-z_j\frac{u_0}{z_0}}{\displaystyle 1-\sum_{l=j+1}^d u_l+\sum_{l=j+1}^d z_l\frac{u_0}{z_0}} \le \frac{u_j}{\displaystyle 1-\sum_{l=j+1}^d u_l}\le t.
$$
It implies that $v=G_z^{-1}(u)\in K_0$ for each $u\in K$ and $z\in V_0$, where we denote 
\begin{align}\label{K0}
K_0=\left\{ (v_1,\dots,v_d)\in\mathcal{S}_d: s(1-t)^{d-1}-\delta\le\frac{v_j}{1-\sum_{l=j+1}^d v_l}\le t, j=1,2,...,d \right\}.
\end{align}
Observe that $T^{-1}(K_0)=[s(1-t)^{d-1}-\delta,t]^d$. Thus,
$T^{-1}\circ G^{-1}_z(K) \subset [s(1-t)^{d-1}-\delta,t]^d$.\\

\vspace{2mm}
\noindent
(b) For $u\in K, z\in V_k$, let
\begin{align*}v&=(v_1,v_2,\dots,...,v_k):=G_{R_k(z)}^{-1}(R_k(u))\\
&=\left(u_{0}-z_{0}\frac{u_k}{z_k}, u_1-z_1\frac{u_k}{z_k},\dots,u_{k-1}-z_{k-1}\frac{u_k}{z_k},u_{k+1}-z_{k+1}\frac{u_k}{z_k},\dots, u_{d}-z_{d}\frac{u_k}{z_k}\right).
\end{align*}
Note that $z_l\le 1-z_k \le \delta$ for $l\in\{0,2,...,d\}\setminus \{k\}$ and $u_{k} \le {\max\{u_d,1-u_d\}} \le {1-\delta}\le {z_k}.$
Therefore, we observe that\\

(i) for $k+1\le j\le d$,
$$\frac{v_j}{\displaystyle 1-\sum_{l=j+1}^d v_l}=\frac{\displaystyle u_j-z_j\frac{u_k}{z_k}}{\displaystyle 1-\sum_{l=j+1}^d u_l+\sum_{l=j+1}^d z_l\frac{u_k}{z_k}} \le \frac{u_j}{\displaystyle 1-\sum_{l=j+1}^d u_l}\le t$$
and 
\begin{align*}
\frac{v_j}{\displaystyle 1-\sum_{l=j+1}^d v_l}
& \ge v_j = u_j-z_j\frac{u_k}{z_k}\ge u_j-z_j\ge s(1-t)^{d-1}-\delta.
\end{align*}

(ii) for $j=1$, we have
\begin{align*}
\frac{v_{1}}{\displaystyle 1-\sum_{l=2}^{d} v_l }& =\frac{\displaystyle u_0-z_0\frac{u_k}{z_k}}{\displaystyle u_0+\sum_{l=1}^d z_l\frac{u_k}{z_k}}=1-\frac{{u_k}}{\displaystyle \left(1-\sum_{l=1}^d u_l\right)z_k+{u_k}\sum_{l=1}^d z_l}\\
& \le 1-\frac{u_k} {\displaystyle \left(1-\sum_{l=k}^d u_l\right)z_k+{u_k}}
 \le 1-\frac{u_k} {\displaystyle 1-\sum_{l=k+1}^d u_l }\le 1-s
\end{align*}
and
\begin{align*}
& \frac{v_{1}}{\displaystyle 1-\sum_{l=2}^{d} v_l } \ge v_1= \displaystyle u_0-z_0\frac{u_k}{z_k}\ge (1-t)^d-\delta.
\end{align*}
(iii) for $2\le j\le k$,
$$
s(1-t)^{d-1}-\delta \le v_j \le \frac{v_j}{\displaystyle 1-\sum_{l=j+1}^{d} v_l }=\frac{\displaystyle u_{j-1}-z_{j-1}\frac{u_k}{z_k}}{\displaystyle 1-\sum_{l=j}^d u_l+\sum_{l=j}^d z_l\frac{u_k}{z_k}} \le \frac{u_{j-1}}{\displaystyle 1-\sum_{l=j}^d u_l}\le t. 
$$ 

Therefore,
$$v\in T\left([(1-t)^{d}-\delta,1-s]\times[s(1-t)^{d-1}-\delta,t]^{d-1}\right). $$

\section*{Acknowledgement}
We would like to thank the anonymous referees for very careful reading of our manuscript and their useful comments, which substantially improved the paper. We also would like to thank Andre R.~Wade for useful suggestions. SV research is partially supported by the grants from Swedish Research Council (VR2014-5147) and the Crafoord Foundation.


\begin{thebibliography}{1}

\bibitem{Borovkov}
{\sc Borovkov, A. A.} (1998). \textit{Ergodicity and Stability of Stochastic Processes}. Wiley, New York.

\bibitem{DeGroot}
{\sc DeGroot, M. H. and Rao, M. M.} (1963). Stochastic give-and-take. \textit{J. Math. Anal. Appl.} {\bf 7}, 489--498.

\bibitem{Diaconis}
{\sc Diaconis, P. and Freedman, D.} (1999). Iterated random functions. \textit{SIAM Rev.} {\bf 41}, no. 1, 45--76.

\bibitem{Hitczenko}
{\sc Hitczenko, P. and Letac, G.} (2014). Dirichlet and quasi-Bernoulli laws for perpetuities. \textit{J. Appl. Probab.} {\bf 51}, no. 2, 400--416. 

\bibitem{simplex}
{\sc Hofrichter, J, Jost, J, Tran, T.} (2017).
\textit{Information Geometry and Population Genetics: The Mathematical Structure of the Wright-Fisher Model}. Springer, Cham.

\bibitem{Ladjimi}
{\sc Ladjimi F. and Peign\'e. M.} (2019). On the asymptotic behavior of the Diaconis-Freedman chain on [0,1]. \textit{Statist. Probab. Lett.} {\bf 145}, no. 2, 1--11.  

\bibitem{McKinlay}
{\sc McKinlay, S. and Borovkov, K.} (2016). On explicit form of the stationary distributions for a class of bounded Markov chains. \textit{J. Appl. Probab.} {\bf 53}, no. 1, 231--243.

\bibitem{Pacheco}
{\sc Pacheco-Gonz\'alez, C. G.} (2009). Ergodicity of a bounded Markov chain with attractiveness towards the centre. \textit{Statist. Probab. Lett.} {\bf 79}, no. 20, 2177--2181.

\bibitem{Ramli}
{\sc Ramli, M. A. and Leng, G.} (2010). The stationary probability density of a class of bounded Markov processes. \textit{Adv. in Appl. Probab.} {\bf 42}, no. 4, 986--993.

\bibitem{Sethuraman}
{\sc Sethuraman, J.} (1994). A constructive definition of Dirichlet priors. \textit{Statist. Sinica} {\bf 4}, no. 2, 639--650. 
\end{thebibliography}
\end{document}